\documentclass[11pt, a4paper]{article}
\usepackage{amssymb, amsmath, amsthm, amsfonts,enumerate}
\usepackage{amsmath,setspace,scalefnt}
\usepackage{fancyhdr}
\usepackage[british]{babel}

\usepackage[usenames,dvipsnames,table]{xcolor}
\usepackage{tikz}
\usetikzlibrary{decorations.pathreplacing}
\usetikzlibrary{decorations.pathmorphing}
\usetikzlibrary{snakes}
\usepackage{bbm}

\usepackage{kotex}
\usepackage{todonotes}

\usepackage{hyperref}
\usepackage[margin=2.75cm]{geometry}
\usepackage{enumitem,verbatim}

\newenvironment{poc}{\begin{proof}[Proof of claim]}{\end{proof}}

\definecolor{darkblue}{rgb}{0,0,0.5}

\newtheorem{theorem}{Theorem}[section]
\newtheorem{lemma}[theorem]{Lemma}

\newtheorem{conj}[theorem]{Conjecture}

\theoremstyle{definition}

\theoremstyle{plain}
\newtheorem{claim}[theorem]{Claim}

\theoremstyle{definition}

\theoremstyle{definition}

\theoremstyle{definition}

\theoremstyle{definition}
\newtheorem{defn}[theorem]{Definition}
\theoremstyle{definition}

\theoremstyle{definition}

\theoremstyle{definition}

\newcommand{\ep}{\varepsilon}

\newcommand{\bN}{\ensuremath{\mathbb{N}}}

\title{Shape of the asymptotic maximum sum-free sets in integer lattice grids}

\author{
	Hong Liu\thanks{Extremal Combinatorics and Probability Group (ECOPRO), Institute for Basic Science (IBS), Daejeon, South
Korea, Email: {\tt hongliu@ibs.re.kr}. H.L. was supported by the Institute for Basic Science (IBS-R029-C4).}
	\and
	Guanghui Wang\thanks{School of Mathematics, Shandong University, China. Email: {\tt ghwang@sdu.edu.cn}. G.W. was supported by National Key R\&D Program of China (2020YFA0712400), Natural Science Foundation of China (11871311, 11631014) and seed fund program for international research cooperation of Shandong University.}
	\and
	Laurence Wilkes\thanks{Department of Computer Science, KU Leuven, Belgium. Email: {\tt laurence.wilkes@kuleuven.be}.
	}
	\and
	Donglei Yang\thanks{Data Science Institute, Shandong University, China, Email: {\tt dlyang@sdu.edu.cn}. D.Y. was supported by the China Postdoctoral Science Foundation (2021T140413), Natural Science Foundation of China (12101365) and Natural Science Foundation of Shandong Province (ZR2021QA029).}
}

\begin{document}

\maketitle

\begin{abstract}
    We determine the shape of all sum-free sets in $\{1,2,\ldots,n\}^2$ of size close to the maximum $\frac{3}{5}n^2$, solving a problem of Elsholtz and Rackham. We show that all such asymptotic maximum sum-free sets lie completely in the stripe $\frac{4}{5}n-o(n)\le x+y\le\frac{8}{5}n+ o(n)$.  We also determine for any positive integer $p$ the maximum size of a subset $A\subseteq \{1,2,\ldots,n\}^2$ which forbids the triple $(x,y,z)$ satisfying $px+py=z$.
\end{abstract}

\section{Introduction}
A cornerstone result of Schur \cite{SCHUR} states that for sufficiently large integer $n$ and a fixed integer $r$, any $r$-coloring of $[n] := \{1,2,\ldots, n\}$
yields a monochromatic triple $x, y, z$ such that $x + y = z$. For an integer $n\in \mathbb{N}$ a subset $A \subseteq [n]$ is \emph{sum-free} if it has no solution for the equation $x+y=z$, i.e. for all $x,y \in A$ we have $x+y \notin A$. The topic of sum-free sets of integers is well-studied in combinatorial number theory and has a long history.

It is clear that the sets
\begin{align*}
S_1 &=\left\{1,3,5,\ldots,2\left\lfloor\frac{n-1}{2}\right\rfloor+1\right\}\quad \text{ and } \quad S_2 =\left\{\left\lceil\frac{n+1}{2}\right\rceil,\left\lceil\frac{n+1}{2}\right\rceil+1,\ldots,n\right\}
\end{align*}
are sum-free and of size $\left\lceil \frac{n}{2} \right\rceil.$ If $n$ is even, $S_3 =S_2-1$ is another one of the same size. Let us denote the density of a maximum sum-free subset of $[n]$ by $\mu([n]):=\max\{\frac{|S|}{n}\mid S\subseteq [n], \textit{S is sum-free}\}$. If $S\subseteq [n]$ is a sum-free set and $a\in S$ is the largest element, then at most one of $x$ or $a-x$ can be in $S$ for each $x\leq a.$ Therefore $|S|\leq \left\lceil \frac{a}{2} \right\rceil \leq \left\lceil \frac{n}{2} \right\rceil$. Together with the above examples, we see that
\[
\mu([n]) = \left\{
\begin{array}{ll}
      \frac{1}{2} &\text{if $n$ even}, \\[1ex]
      \frac{1}{2}+\frac{1}{2 n}  &\text{if $n$ odd}. \\
\end{array}
\right.
\]
\subsection{Structure for large sum-free sets}
Given the extremal result, great efforts has been made to better understand
the general structure of large sum-free sets in $[n]$. The first result on this topic was due to Freiman \cite{Freiman92} who showed that if the size of a sum-free set in $[n]$ is large enough, then it will either consist of all odd numbers as in $S_1$ above or it will be close to the second half of the interval as $S_2$. We remark that more structural results are
known for large sum-free sets in the $1$-dimensional integer lattice (see \cite{DE99} and a recent progress~\cite{Tuan18}). Such structural results are not only interesting on their own; they have been utilized e.g. in recent work on enumerating maximal sum-free sets (see \cite{BLST}).

The problem of sum-free sets has been generalized to higher dimensional lattice $\mathbb{Z}^d$, $d\ge2$. Similarly, we define
$\mu([n]^d):=\max\{\frac{|S|}{n^d}\mid S\subseteq [n]^d \textit{ is sum-free} \}$. In particular, for $d=2$, the problem of finding the largest sum-free subset of $[n]^2=\{1, 2,\ldots, n\}^2$ was firstly presented by Cameron as an
unsolved problem in \cite{cameron19}.

\begin{conj}\emph{\cite{cameron19}}\label{conj1}
  There exists a constant $c\in \mathbb{R}$ such that $\mu([n]^2)=c+O(1/n).$
\end{conj}
Cameron later \cite{cam2} suggested that Conjecture \ref{conj1} is true with $c=0.6$ and gave a lower bound construction: $$S_0=\{(x,y)\in[n]^2\mid u\le x+y\le 2u-1\},$$ which has maximum density $0.6$ when  $u=\lfloor\frac{4n+7}{5}\rfloor$. Recently, Elsholtz and Rackham settled Conjecture \ref{conj1} in \cite{E-R}, proving that indeed
  $$\mu([n]^2)=0.6+O(1/n).$$
In the same paper, Elsholtz and Rackham \cite{E-R} raised the problem of classifying the sum-free sets whose size are close to the extremal value.

In this paper, we resolve this problem by showing that any sum-free subset $S\subseteq [n]^2$ of size at least $(\frac{3}{5}-o(1))n^2$ will have all its points in the region $\{(x,y) \in [n]^2$ $|$ $\frac{4n}{5}-o(n) \leq x+y < \frac{8n}{5}+o(n)\}.$

\begin{theorem}\label{mainthm}
For all $\gamma > 0$ there exists $\delta > 0$ and $n_0 \in \mathbb{N}$ such that the following holds for all $n>n_0$. If $S \subseteq [n]^2$ is sum-free with $|S| > (\frac{3}{5}-\delta)n^2$, then
$$ S\subseteq\{(x,y) \in [n]^2\mid
\frac{4n}{5}-\gamma n \leq x+y < \frac{8n}{5}+\gamma n\}.$$
\end{theorem}

This gives a satisfying answer to the 2-dimension sum-free problem. The situation is, however, unclear for higher dimension. In particular, even the maximum density of a sum-free set in the 3-dimension grid $[n]^3$ is unknown.

\subsection{$(p,q)$-sum-free sets}
Given positive integers $d,n$ and rational numbers $p,q$, a set $S\subseteq[n]^d$ is called $(p,q)$-\emph{sum-free} if it has no solution for the equation $px+qy=z$. As a generalization of sum-free sets (i.e. (1,1)-sum-free sets), the notion of $(p,q)$-sum-free sets encapsulates many fundamental topics in combinatorial number theory.  In particular, for $d=1$, a $(\frac{1}{2},\frac{1}{2})$-sum-free set is precisely a set without $3$-term arithmetic progression, which has received considerable attention in recent decades.
Therefore, it is a natural question to determine the size of the largest $(p,q)$-sum-free sets in $[n]^d$. Here one can similarly define $$\mu_{[p,q]}([n]^d):=\max\left\{\frac{|S|}{n^d}\mid S\subseteq [n]^d \textit{ is $(p,q)$-sum-free}\right\}.$$
By Roth's theorem \cite{roth53}, $\mu_{[1/2,1/2]}([n])=o(1)$. See \cite{bloom16} for the best known upper bound for the size of a $(1/2,1/2)$-sum-free set. In \cite{ru1,ru2}, instead of the form $x+y=z$, Ruzsa instigated the study of more general linear equations $a_1x_1 +\cdots + a_kx_k = b$. In particular, for more general cases when $p,q$ are positive integers and $p\ge2$, Hancock and Treglown \cite{HAN17} completely determined the value $\mu_{[p,q]}([n])$. For higher dimensional lattices, Choi, Kim and Park \cite{cho} initiated the investigation of the form $x_1+x_2+\cdots+x_k=b$, where $b$ is a prescribed point in $[n]^2$.

For 2-dimension $(p,q)$-sum-free problem, we make the first attempt to determine $\mu_{[p,p]}([n]^2)$ for any integer $p$.
\begin{theorem}\label{thm2}
  Let $p\in\bN$ and $S\subseteq [n]^2$ be a $(p,p)$-sum-free set. Then $$|S|\leq\Big(1-\frac{2}{4p^2+1}\Big)n^2+O(n).$$
\end{theorem}

We observe that the upper bound in Theorem \ref{thm2} is optimal up to the error term $O(n)$, given by the following construction. For any positive integers $p,q$ and positive real $a$, define $S=\{(x,y)\in[n]^2\mid   a<x+y<(p+q)a\}$. One can easily check that $S$ is $(p,q)$-sum-free with size $$|S|=\left(1-\frac{2}{(p+q)^2+1}\right)n^2+O(n),$$
when $a=\frac{2(p+q)}{(p+q)^2+1}n$, corresponding to the stripe $$S=\left\{(x,y)\in[n]^2\mid   \frac{2(p+q)}{(p+q)^2+1}n<x+y<\frac{2(p+q)^2}{(p+q)^2+1}n\right\}.$$

We conjecture that for all integers $p$ and $q$, the above construction provides the maximum $(p,q)$-sum-free set.
\begin{conj}
Let $p,q$ and $n$ be positive integers and $S\subseteq [n]^2$ be a $(p,q)$-sum-free set. Then $$|S|\leq\left(1-\frac{2}{(p+q)^2+1}\right)n^2+O(n).$$
\end{conj}

\medskip

\noindent\textbf{Organization.} The rest of the paper will be organized as follows. Section~\ref{sec-prelim} includes some notation and tools needed. Section~\ref{sec-main} is devoted to the proof of Theorem \ref{mainthm}. The proof of Theorem \ref{thm2} is given in
Section~\ref{sec-pqSF}.


\section{Preliminaries}\label{sec-prelim}
Given a convex polygon $P$ in $\mathbb{R}^2_{[0,n]}$, denote by $\Lambda(P)$ the number of lattice points contained within $P$, and by $\|P\|$ the area of $P$ with respect to the Lebesgue measure. The \emph{translate} of $P$ by a vector $a\in \mathbb{R}^2_{[0,n]}$ is denoted as $P+a:=\{a+(x,y)\mid (x,y)\in P\}$. Write $a-P:=\{a-(x,y)\mid (x,y)\in P\}$. Throughout the proof, we always use the following result which is a corollary of Lemma 3.1 in \cite{E-R}.
\begin{lemma}\label{lem-point-area}
If $P$ is a convex polygon in $\mathbb{R}^2_{[0,n]}$ with finitely many sides, then $\Lambda(P)=\|P\|+O(n)$.
\end{lemma}
This lemma implies that any convex polygon $P$, described above, satisfies that $\Lambda(P)=\|P\|+O(n)$, which allows us to focus on the area $\|P\|$ instead of $\Lambda(P)$.

For two points $p_1,p_2\in \mathbb{R}^2_{[0,n]}$, denote by $m(p_1,p_2)$ the gradient and by $c(p_2,p_2)$ the $y$-intercept of the line in $\mathbb{R}^2$ passing through $p_1$ and $p_2.$

\begin{defn}[Upper boundary]\label{upper}
 Given a set $A\subseteq \mathbb{R}^2_{[0,n]}$, the \textit{upper boundary} of $A$ is a set of points in $A$, denoted by $\partial A$, such that for each $p_1\in \partial A$ there exists a point $p_2\in A\setminus \{p_1\}$ with the following properties:
\begin{itemize}
    \item $m(p_1,p_2)<0$;
    \item Let $T = \{(x,y)\subseteq \mathbb{R}^2_{[0,n]}\mid  y>m(p_1,p_2)x+c(p_1,p_2)\}$. Then $|A\cap T|=0$.
\end{itemize}
Any two such points $p_1,p_2$ are said to be \textit{adjoint}, and the line passing through two points that are adjoint is called an \textit{upper boundary line}. The second condition above states that there is no point of $A$ strictly above any upper boundary line.
\end{defn}

The following lemma shows that if the upper boundary of a set $A$ is empty, then $A$ has a `top right corner'.

\begin{lemma}[Lemma 5.1 in \cite{E-R}]\label{onepoint}
Suppose $A\subseteq \mathbb{R}^2_{[0,n]}$ such that $\partial A=\emptyset$. Then there is a point $(a,b)\in A$ such that $a\geq x$ and $b \geq y$ for all $(x,y) \in A.$
\end{lemma}

We also need the concept of pairing sets, which will be frequently used throughout the proof.
\begin{defn}
Given a point $(a_1,a_2)\in \mathbb{R}^2_{[0,n]}$ and a set $P\subseteq \mathbb{R}^2_{[0,n]}$, we call $P$ a \emph{pairing set for} $(a_1,a_2)$ if for any $x\in P$, we have $(a_1,a_2)-x\in P$.
\end{defn}
The following lemma guarantees that for any point in a sum-free set $S$, every pairing set for that point cannot intersect too much with $S$.
\begin{lemma}[Lemma 3.4 in \cite{E-R}]\label{lem-pairing-set}
Let $S$ be a sum-free set in $[n]^2$. Then for any $a\in S$ and a pairing set $P$ for $a$, we have $|P\cap S| \leq \frac{1}{2}\Lambda(P)$.
\end{lemma}

The following lemma bounds the intersection of a set and its translate with a sum-free set.
\begin{lemma}\label{lem-pairing}
Given two sets $S,T\subseteq [n]^2$, if $S$ is sum-free, then for any $a\in S$, it holds that
\[ |S \cap (T \cup (a\pm T))| \leq |T|.\]
\end{lemma}
\begin{proof}
For each element $t\in T$ there is a corresponding element $a\pm t \in a\pm T$. Since $a\in S$, one can observe from sum-freeness that at most one of $t$ and $a\pm t$ belongs to $S$.
\end{proof}

\section{Proof of Theorem \ref{mainthm}}\label{sec-main}
We carry out the proof in a few steps. First, using Lagrange multiplier, we show that any almost maximum-size sum-free set $S$ in $[n]^2$ has an upper boundary line that is close to the line $y+x=\frac{8n}{5}$, see Lemma~\ref{lem-upper-boundary-line}. Then we show that there is a point $(x^*,y^*)$ in $S$ close to $(\frac{4n}{5},\frac{4n}{5})$, see Lemma~\ref{lem-top-right-point}. Finally, using this point $(x^*,y^*)$, we show in Section~\ref{sec-things-together} that $S$ has no point below the line $y+x=\frac{4n}{5}-o(n)$, which, together with the upper boundary line close to $y+x=\frac{8n}{5}$, implies that $S$ must be close to the extremal stripe $\frac{4n}{5}\le x+y\le \frac{8n}{5}$.

Throughout the proofs, when we write $\beta\ll\gamma$, we always mean that $\beta,\gamma$ are constants in $(0, 1)$, and there exists $\beta_0= \beta_0(\gamma)$ such that the subsequent arguments hold for all $0<\beta \le \beta_0$. Hierarchies of other lengths are defined analogously.

\begin{defn}\label{def1}
	A sum-free set $S\subseteq [n]^2$ with $\partial S\neq \emptyset $ is of \emph{Type} 1 if there exists a point $p_1=(x_1,y_1)\in \partial S$ with $x_1\leq y_1$ and $x_1 y_1 \geq x y$ for all $(x,y) \in \partial S$, and a point $p_2=(x_2,y_2)$ adjoint to $p_1$ satisfying the following conditions, where we simply write $m=m(p_1,p_2)$ and $c=c(p_1,p_2)$.
	\begin{itemize}
		\item [$(1)$] $x_2>x_1$, $y_2<y_1$ and $m<-\frac{y_1}{x_1}\leq -1$;
		\item [$(2)$] $c>n$ and $-c\leq n m$.
	\end{itemize}
In addition, $S$ is of \emph{Type} 2 if there exist two adjoint points $p_1=(x_1,y_1)$ and $p_2=(x_2,y_2)$ in $\partial S$ satisfying the following conditions.
     \begin{itemize}
     	\item [$(1)$] $x_2>x_1$, $y_2<y_1$ and $-\frac{y_1}{x_1}\leq m\leq -\frac{y_2}{x_2}$;
     	\item [$(2)$] $y_2\leq \frac{c}{2} \leq y_1$;
     	\item [$(3)$] $c>n$ and $-c< n m$.
     \end{itemize}
\end{defn}

For either type of the sum-free sets, we call the upper boundary lines passing through $p_1$ and $p_2$ \emph{typical}.
Let
$$A=\{(x,y)\in \mathbb{R}^2_{[0,n]}\mid y>mx+c\}$$
with $m$ and $c$ given as above. Then $A$ is a triangle in both cases.

For the Type 1 set $S$, we claim that the upper boundary line $y=mx+c$ satisfies $x_1>\frac{n}{2}$. In fact, since $m<-\frac{y_1}{x_1}$ and $y_1=mx_1+c$, we have that $x_1>\frac{c}{-2m}>\frac{n}{2}$ because $-c<n m$.

If $S$ is of Type 2, then it is straightforward to check that the following two sets are nonempty (see Figure~\ref{T2}).
 \[T_1=\left \{(x,y) \in \mathbb{R}^2_{[0,n]} \mid   x \geq x_1,y-mx \leq \frac{c}{2} \right \},\] \[T_2=\left \{(x,y) \in \mathbb{R}^2_{[0,n]} \mid   y \geq y_2,y-mx \leq \frac{c}{2} \right \}.\]

\begin{figure}[h]
  \centering
  \includegraphics[width=180pt]{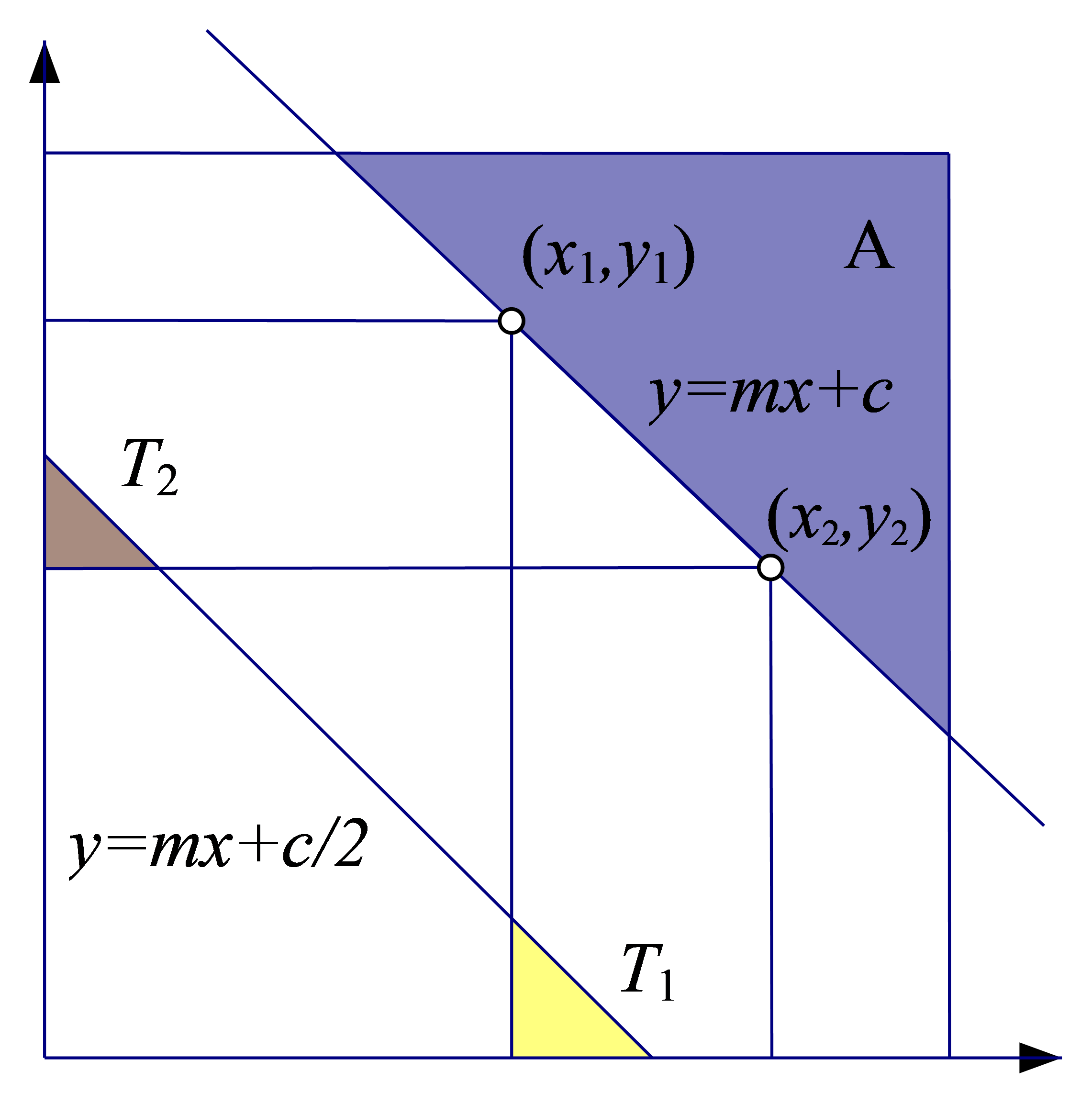}\\
  \caption{$S$ is of Type 2}\label{T2}
\end{figure}

The two types we defined above correspond to the only two cases in~\cite{E-R} that attain the maximum density $\frac{3}{5}$. We will use the following bounds for these two types of sum-free sets.

\begin{lemma}[\cite{E-R}]\label{lem-ER-structure}
	Given a sum-free set $S \subseteq [n]^2$, if $|S|>0.56n^2$, then either
	\begin{itemize}
		\item[\emph{(1)}] $S$ is of Type 1 and $|S| \leq (n+1)^2 - \frac{1}{2} x_1 y_1 + \frac{(c+n m-n)^2}{2m}$, or
		
		\item[\emph{(2)}] $S$ is of Type 2 and $|S|\le  (n+1)^2+\frac{c^2}{8m}+\frac{1}{2m}(n-n m-c)^2.$
	\end{itemize}
\end{lemma}

\subsection{Fixing an upper boundary line}
Given constants $\varepsilon$ and $C$, we call a line $L$ $\varepsilon$-\emph{close} to the line $x+y=C$ if the portion of $L$ intersecting $\mathbb{R}^2_{[0,n]}$ lies entirely within the set $\{ (x,y) \in \mathbb{R}^2_{[0,n]} \mid |x+y-C| \le \varepsilon n\}.$ Similarly, we call two points $p_1=(x_1,y_1)$ and $p_2=(x_2,y_2)$ $\varepsilon$-\emph{close} to each other if $|x_1-x_2|\le \varepsilon n$ and $|y_1-y_2|\le \varepsilon n$.
\begin{lemma}
\label{lem-upper-boundary-line}
Given $\varepsilon>0$, there exist $\delta > 0$ and $n_0 \in \mathbb{N}$ such that the following holds for all $n>n_0$. If $S \subseteq [n]^2$ is sum-free and $|S| > (\frac{3}{5}-\delta)n^2$, then there is a typical upper boundary line for $S$ which is $\varepsilon$-close to $x+y=\frac{8n}{5}$.
\end{lemma}

\begin{proof}
Given $\varepsilon>0$, let $\delta=\frac{\varepsilon^2}{100}$ and $n$ be sufficiently large with respect to $\varepsilon$. Let $S \subseteq [n]^2$ be a sum-free set with $|S| > (\frac{3}{5}-\frac{\varepsilon^2}{100})n^2$.

Suppose for contradiction that any upper boundary line $y=mx+c$ for $S$ is not $\ep$-close to $x+y=\frac{8n}{5}$. That is, either the $y$- or the $x$-intercept is far from where it should be: $$\text{either}\quad |c-8n/5|>\varepsilon n\quad \text{or}\quad |c/m+8n/5|>\varepsilon n.$$
In both cases we shall obtain a contradiction by showing that $|S|\leq (3/5-\varepsilon^2/100)n^2$.

 Considering the typical upper boundary line $y=mx+c$ passing through $p_1$ and $p_2$ given in Definition \ref{def1}, we will finish the case when the $y$-intercept is too far, that is, $|c-8n/5|>\varepsilon n$, whose proof will be divided into two cases depending on the type of $S$. The case when the $x$-intercept is too far (that is, $|c/m+8n/5|>\varepsilon n$) is similar and we omit the details.

Suppose first that $S$ is of Type 1, then by Lemma~\ref{lem-ER-structure}(1), we have
\begin{align*}
|S| \leq (n+1)^2-\frac{1}{2}\left(x_1(m x_1+c)- \frac{(c+mn-n)^2}{m}\right)=:f(x,m,c).
\end{align*}
To simplify the presentation, we introduce a new variable $\eta$ with $\eta\in (-\infty, -\varepsilon)\cup(\varepsilon,+\infty)$ and define $$f_{\eta}:=\max\{f(x,m,c)\mid c-8n/5=\eta n\}.$$
Let $L:=f(x,m,c)-\lambda g$, where $g=c-8n/5-\eta n$. By solving $\frac{\partial L}{\partial x}=0$, $\frac{\partial L}{\partial m}=0$, $\frac{\partial L}{\partial c}=0$ and $\frac{\partial L}{\partial \lambda}=0$, we obtain $m=-\sqrt{1+2\eta+\frac{5\eta^2}{4}}$ and $x=\frac{\frac{4}{5}+\frac{\eta}{2}}{\sqrt{1+2\eta+\frac{5\eta^2}{4}}}n$, and thus the maximum value is $$f_{\eta}=\left(8/5+\eta-\sqrt{1+2\eta+5\eta^2/4}\right)n^2+O(n).$$
As $\eta$ takes values over $(-\infty, -\varepsilon)\cup(\varepsilon,+\infty)$, we get $$f_{\eta}\le\left(8/5+\varepsilon-\sqrt{1+2\varepsilon+5\varepsilon^2/4}\right)n^2+O(n)\le (3/5-\varepsilon^2/100)n^2.$$

For the second case when $S$ is of Type 2, by Lemma~\ref{lem-ER-structure}(2), we have:
$$|S| \leq (n+1)^2 + \frac{c^2}{8m} + \frac{(n-n m-c)^2}{2m}.$$
Using Lagrange multiplier again, we arrive at the same bound
$\left(8/5 + \varepsilon - \sqrt{1 +2 \varepsilon + 5 \varepsilon^2/4}\right)n^2+O(n)\le (3/5-\varepsilon^2/100)n^2$ as desired.
\end{proof}

\subsection{Top right corner}
\begin{lemma}
\label{lem-top-right-point}
For any $\beta > 0$, there exist $\delta> 0$ and $n_0 \in \mathbb{N}$ such that for all $n>n_0$, if $S \subseteq [n]^2$ is sum-free with $|S| > (\frac{3}{5}-\delta)n^2$, then there is a point $(x^*,y^*)\in S$ which is $\beta$-close to the point $(\frac{4n}{5},\frac{4n}{5})$.
\end{lemma}

\begin{proof}
We first handle Type 1 sum-free sets.
Given $\beta>0$, we have constants $\delta=\delta_{\ref{lem-upper-boundary-line}} > 0$ and $n_0 \in \mathbb{N}$ returned from Lemma ~\ref{lem-upper-boundary-line} with $\varepsilon=\beta/6$. Let $S \subseteq [n]^2$ be a sum-free set of Type~1 with $|S| > (\frac{3}{5}-\delta)n^2$. Then Lemma \ref{lem-upper-boundary-line} gives a typical upper boundary line $y=mx+c$ that is $\varepsilon$-close to $x+y=\frac{8n}{5}$ and let $p_1=(x_1,y_1)$, $p_2=(x_2,y_2)$ be the two points involved. Therefore, $|c-\frac{8n}{5}|<\varepsilon n$, $|x_1+y_1-\frac{8n}{5}|<\varepsilon n$. Consequently, by triangle inequality we have
$$|m+1|=\frac{|x_1+y_1-c|}{x_1}<\frac{2\varepsilon n}{x_1}<4\varepsilon,$$ where the last inequality follows since $x_1>n/2$. Recall that $m\leq -\frac{y_1}{x_1} \leq -1$. Then we have that $|m+\frac{y_1}{x_1}|<4\varepsilon$.

Using these facts we can write $m=-\frac{y_1}{x_1}-\gamma_1$ and $c=(\frac{8}{5}+\gamma_2)n$ for constants $0 \leq \gamma_1 < 4\varepsilon$ and $|\gamma_2|<\varepsilon$. Using the equation $y_1=m x_1+c$, we obtain that $y_1=\frac{4}{5}n+\frac{\gamma_2 n-\gamma_1 x_1}{2}.$ As $x_1\leq n$, by triangle inequality, we have
$$\Big|y_1-\frac{4n}{5}\Big|<\frac{5\varepsilon n}{2}<\beta n.$$
Moreover, since $-\frac{y_1}{x_1}\ge m>-1-4\varepsilon$ and $x_1\le y_1$, we can easily obtain that $|x_1-\frac{4n}{5}|<6\varepsilon n=\beta n$. So $(x_1,y_1)$ is $\beta$-close to the point $(\frac{4n}{5},\frac{4n}{5})$ as desired.

Let us turn to Type 2 sum-free sets. Now, given $\beta>0$, choose positive constants $\varepsilon, \delta$ with $\delta\ll\varepsilon\ll \beta$. Let $S$ be a sum-free set of Type 2 with $|S| > (\frac{3}{5}-\delta)n^2$. Then applying Lemma \ref{lem-upper-boundary-line} with $\sqrt{2}\varepsilon$ playing the role of $\varepsilon$ gives a typical upper boundary line $y=mx+c$ passing through $p_1=(x_1,y_1)$ and $p_2=(x_2,y_2)$ (see Figure~\ref{F2}), which is $\sqrt{2}\varepsilon$-close to $x+y=\frac{8n}{5}$. This implies that the line $y=m x + \frac{c}{2}$ is $\frac{\varepsilon }{\sqrt{2}}$-close to $x+y=\frac{4n}{5}$. We may assume for contradiction that $S$ has no points in the region
$$ T_\beta =\left\{(x,y) \in \mathbb{R}^2_{[0,n]} \mid x,y \geq \frac{4n}{5} - \beta n \:, \: \: y + x \leq \frac{8n}{5} \right\}.$$
Redefine the regions as follows:
$$A = \left\{ (x,y) \in \mathbb{R}^2_{[0,n]}\mid y+x \geq \frac{8n}{5}+\sqrt{2} \varepsilon n \right\},$$
$$T_1=\left \{(x,y) \in \mathbb{R}^2_{[0,n]} \mid y + x \leq \frac{4n}{5} - \frac{\varepsilon n}{\sqrt{2}},\: x \geq x_1 \right \},$$
$$T_2=\left \{(x,y) \in \mathbb{R}^2_{[0,n]} \mid y + x \leq \frac{4n}{5} - \frac{\varepsilon n}{\sqrt{2}},\: y \geq y_2 \right \}.$$
Note that
$$T_1+T_2=\left\{(x,y)\in \mathbb{R}^2_{[0,n]} \mid y+x \leq \frac{8n}{5}-\sqrt{2}\varepsilon n,\: x \geq x_1,\: y \geq y_2\right\}.$$
\begin{figure}
  \centering
  \includegraphics[width=200pt]{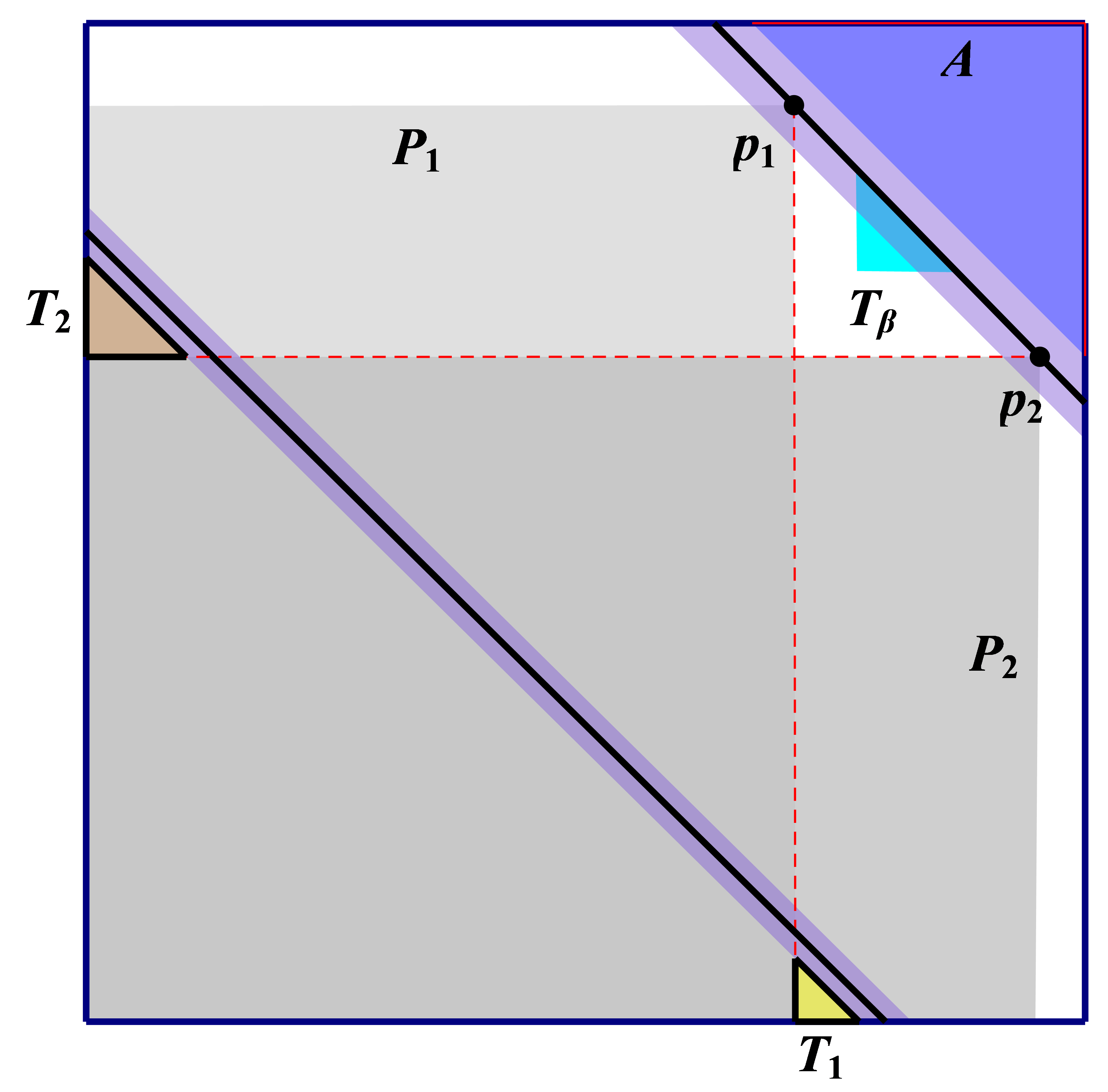}\\
  \caption{$S$ is of Type 2: the two purple stripes are $\{(x,y)\mid |x+y-\frac{8n}{5}|\le\sqrt{2}\varepsilon n\}$ (on the top right) and $\{(x,y)\mid |x+y-\frac{4n}{5}|\le\frac{\varepsilon n}{\sqrt{2}}\}$.}\label{F2}
\end{figure}
We now proceed by considering the areas which may be excluded from $S$. Firstly, we show that $S$ has two points in $T_1$ that are far apart.
\begin{claim}\label{cl1}
  There are two points in $T_1\cap S$ which are at least $\beta n$ far apart.
\end{claim}

\begin{poc}
If this is not true, then there are less than $\pi(\frac{\beta n}{2})^2\le\beta^2 n^2$ points of $S$ in $T_1$, given by the number of points in a square around a circle of diameter $\beta n$ in $T_1$. Since $\|T_{\beta}\|=\frac{1}{2}(2\beta)^2 n^2=2\beta^2 n^2$, we then use the pairing set $P_1$ for $(x_1,y_1)$ and thus
\begin{alignat*}{4}
    |S| &\leq n^2 - \Lambda(A) - \frac{1}{2}\Lambda(P_1) - \Lambda(T_1) +\beta^2 n^2 - \Lambda(T_\beta) \\
    &=n^2 - \|A\| - \frac{1}{2}\|P_1\| - \|T_1\| +\beta^2 n^2- \|T_\beta\| + O(n) \\
    &= n^2 - \frac{1}{2}\left(\frac{2}{5}-\sqrt{2}\varepsilon \right)^2 n^2 - \frac{1}{2} x_1 y_1 - \frac{1}{2}\left(\frac{4}{5}-\frac{\varepsilon}{\sqrt{2}} -\frac{x_1}{n}\right)^2 n^2 +\beta^2 n^2 -2\beta^2 n^2 +O(n).
\end{alignat*}
It is easy to see this is maximized when $y_1$ is minimal and $x_1+y_1=\frac{8n}{5}-\sqrt{2}\varepsilon n$. Then
\begin{align*}
    |S| &\leq \left(\frac{3}{5}-\beta^2+10\varepsilon \right)n^2.
\end{align*}
Therefore, we reach a contradiction by the fact that $\delta\ll\varepsilon\ll \beta$.
\end{poc}

By Claim \ref{cl1}, we let $s$ and $t$ be two points in $T_1$ with distance greater than $\beta n$, and let
$$T_2^s:=s+T_2 \quad \text{and} \quad T_2^t:=t+T_2.$$

\begin{claim}\label{cl2}
  $\Lambda(T_\beta \setminus T_2^s), \Lambda(T_\beta \setminus T_2^t)<\frac{\beta^2}{L} n^2$, where $L=\frac{4}{5\sqrt{3}-6}$.
\end{claim}

\begin{poc}
Suppose to the contrary that either $\Lambda(T_\beta \setminus T_2^s) \geq \frac{\beta^2}{L} n^2$ or $\Lambda(T_\beta \setminus T_2^t) \geq \frac{\beta^2}{L} n^2$, and by symmetry we may assume the first inequality holds. Considering the pairing set $P_2$ for $(x_2,y_2)$ and $T_2$ paired with $T_2^s$, we can obtain from Lemmas \ref{lem-pairing-set} and \ref{lem-pairing} that
\begin{align*}
    |S| &\leq n^2 - \Lambda(A) - \frac{1}{2}\Lambda(P_2) - \Lambda(T_2) - \Lambda(T_\beta \setminus T_2^s) \\
    &\leq  \left(\frac{3}{5}-\frac{\beta^2}{L}+O(\varepsilon) \right)n^2,
\end{align*}
which once again gives a contradiction as $\delta\ll\beta$.
\end{poc}

In the rest of the proof, we shall find a partition $T_2=T_{2,1}\cup T_{2,2}$ into two regions such that their corresponding translates $T^s_{2,1}=s+T_{2,1}$ and $T^t_{2,2}=t+T_{2,2}$ are distantly separated in $T_1+T_2$, which provides a significant portion of points in $T_{\beta}\setminus(T_{2,1}^s\cup T_{2,2}^t)$ that are to be excluded from $S$.

Write $s=(x_s,y_s)$ and $t=(x_t,y_t)$. By Claim \ref{cl2}, we can find that the two points $s+(0,y_2)$ and $t+(0,y_2)$ belong to the region $\{(x,y)\in [n]^2\mid x,y\le \frac{4n}{5}\}$. We may assume $x_s+y_s \geq x_t+y_t$ and let $d:=x_s+y_s-(x_t+y_t)$. It is easy to see in Figure \ref{fig:claim3} that $d$ is the difference between the corresponding $y$-intercepts of the red diagonal and the blue diagonal. By the symmetry of all the shapes involved, we can further assume that $x_s\ge x_t$.

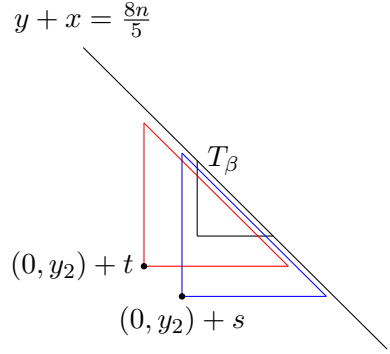
\begin{figure}[h]
    \centering
    \begin{tikzpicture}
        \draw (4,0) -- (0,4) node[anchor=south] {$y+x=\frac{8n}{5}$};
        \draw (1.5,1.5) -- (1.5,2.5) node[anchor=west] {$T_\beta$};
        \draw (1.5,1.5) -- (2.5,1.5);
        \draw[red] (0.8,1.1) -- (0.8,3);
        \draw[red] (0.8,1.1) -- (2.7,1.1);
        \draw[red] (0.8,3) -- (2.7,1.1);
        \filldraw (0.8,1.1) circle (1pt) node[anchor=east] {$(0,y_2)+t$};
        \draw[blue] (1.3,0.7) -- (1.3,2.6);
        \draw[blue] (1.3,0.7) -- (3.2,0.7);
        \draw[blue] (1.3,2.6) -- (3.2,0.7);
        \filldraw (1.3,0.7) circle (1pt) node[anchor=north] {$(0,y_2)+s$};
    \end{tikzpicture}
    \caption{ The red triangle represents $T_2^t$, the blue one represents $T_2^s$ and the black one represents $T_\beta$. }
    \label{fig:claim3}
\end{figure}
\begin{claim}\label{cl3}
  $d \leq \left(2-2\sqrt{\frac{2L-1}{2L}}\right)\beta n = \frac{\sqrt{3}-1}{2}\beta n$.
\end{claim}
\begin{poc}
It is easy to see the region $T_\beta \cap T_2^t$ is a triangle similar to $T_\beta$. 
Note that the area of $T_\beta \setminus T_2^t$ is at least
\[\frac{1}{2}(2\beta n)^2-\frac{1}{2}(2\beta n-d)^2=\Big(2\beta n -\frac{d}{2}\Big)d.\]
By Claim \ref{cl2}, we have that $\left(2\beta n -\frac{d}{2}\right)d \leq \frac{\beta^2 n^2}{L}$, which yields the bound on $d$ as desired.
\end{poc}

Define points
$$X_1=(x_s,y_s+y_2), \quad X_2=(x_t,y_t+y_2) \quad \text{and}\quad X_3=(x_s,y_t+y_2).$$
Let $P_1X_1$ and $P_2X_2$ be line segments which are parallel to $PX_3$ (see Figure \ref{F4}).
Construct a line passing through $(0,y_2)$ which is also parallel to the line segments $PX_3$, where $P=(\frac{4n}{5},\frac{4n}{5})$. Such a line separates $T_2$ into two parts, and we denote by $T_{2,2}$ the part above the line and $T_{2,1}$ for the rest.

\begin{figure}[h]
  \centering
  \includegraphics[width=189pt]{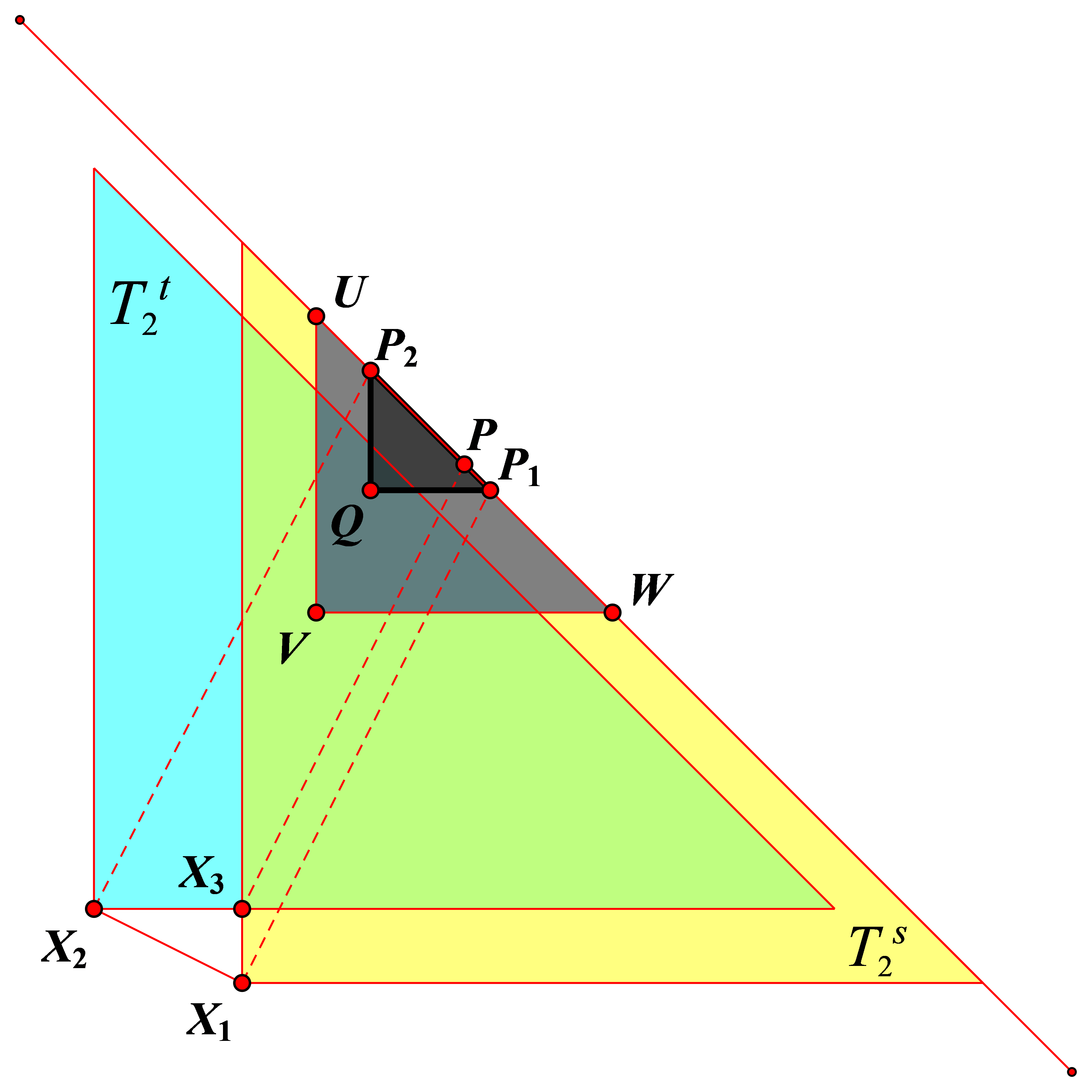}
  \caption{The triangle $UVW$ represents $T_{\beta}$, in which $P=(\frac{4n}{5},\frac{4n}{5})$ is the median point for the line segment $UW$.}\label{F4}
\end{figure}

\begin{claim}\label{cl4}
There exists a triangle $T_{\beta,1} \subseteq T_\beta$ similar to $T_\beta$ such that $T_{\beta,1}$ does not intersect with either of the regions $T_{2,1}+s$ or $T_{2,2}+t$ and $\Lambda(T_{\beta,1})\ge\frac{\beta^2n^2}{8}$.
\end{claim}

\begin{poc}
Let $h:=y_t-y_s$. Since $s$ and $t$ are of distance at least $\beta n$ far apart, that is, $(x_s-x_t)^2+(y_s-y_t)^2=(h+d)^2+h^2\ge\beta^2n^2$, together with Claim~\ref{cl3}, we obtain that either $h\ge\frac{\sqrt{2\beta^2-d^2}}{2}-\frac{d}{2}\ge\frac{\beta n}{2}$ or $h\le-\frac{\sqrt{2\beta^2-d^2}}{2}-\frac{d}{2}<-d$, where the latter contradicts with the assumption that $x_s-x_t=h+d\ge0$. Thus, $h\ge\frac{\beta n}{2}$ and the segment $P_1P_2$ has length at least $\frac{\sqrt{2}}{2}\beta n$. Let $T_{\beta,1}$ be the rectangle triangle $P_1P_2Q$ with diagonal line $P_1P_2$. Then $T_{\beta,1}$ has area at least $\frac{\beta^2n^2}{8}$ and does not intersect either of the regions $T_{2,1}+s$ or $T_{2,2}+t$.
\end{poc}

As aforementioned, now we are ready to finish the proof. Applying Lemma \ref{lem-pairing} to  $T_{2,1}, T_{2,2}$ and their translates $T_{2,1}+s$, $T_{2,2}+t$, we obtain that
\begin{align*}
    |S| &\leq n^2 - \Lambda(A) - \frac{1}{2}\Lambda(P_2) - \Lambda(T_2) - \Lambda(T_{\beta,1})  \\
    &= n^2 - \|A\| - \frac{1}{2}\|P_2\| - \|T_2\| - \|T_{\beta,1}\| +O(n) \\
    &\leq n^2 - \frac{1}{2}\left(\frac{2}{5}-\sqrt{2}\varepsilon \right)^2 n^2 - \frac{1}{2} x_2 y_2
     - \frac{1}{2}\left(\frac{4}{5}-\frac{\varepsilon}{\sqrt{2}}-\frac{x_2}{n}\right)^2 n^2 - \frac{\beta^2 n^2}{8} +O(n).
\end{align*}
The right-hand side above is maximized when $y_2$ is minimal and $x_2+y_2=\frac{8n}{5}-\sqrt{2}\varepsilon n$. Thus,
\begin{align*}
    |S| &\leq \left(\frac{3}{5}-\frac{\beta^2}{8}+O(\varepsilon) \right)n^2,
\end{align*}
a final contradiction.
\end{proof}

\subsection{Putting things together}\label{sec-things-together}
We are now ready to prove our main result, knowing that any almost maximum sum-free set contains an upper boundary line $o(1)$-close to $y+x=\frac{8n}{5}$ and a point $o(1)$-close to $(\frac{4n}{5},\frac{4n}{5})$.

\begin{proof}[Proof of Theorem \ref{mainthm}]
Given $\gamma>0$, choose $\delta\ll \varepsilon\ll\beta\ll\gamma$. Let $S\subseteq[n]^2$ be a sum-free set of size at least $(3/5-\delta)n^2$. Then by Lemma~\ref{lem-upper-boundary-line}, $S$ has a typical upper boundary line $y=mx+c$ which is $\varepsilon$-close to $y+x=\frac{8n}{5}$. Now it suffices to show that $S$ has no point below the line  $x+y=\frac{4n}{5}-\gamma n$ (see the red line in Figure \ref{fig6}).

Note that Lemma \ref{lem-top-right-point} ensures the existence of a point $(x_1,y_1)$ in $S$ that is $\beta$-close to $(\frac{4n}{5},\frac{4n}{5})$. Suppose to the contrary that $p_0=(x_0,y_0)\in S$ is such a point below the line $x+y=\frac{4n}{5}-\gamma n$, and without loss of generality we may assume that $y_0\ge x_0.$

Let
\[A=:\left\{ (x,y) \in [n]^2 \mid y+x>\frac{8n}{5}+\varepsilon n \right\}.\]
Considering the pairing set $P:=\{(x,y) \in [n]^2 \mid x\leq x_1, y \leq y_1 \}$ for $(x_1,y_1)$, there are at most
\begin{align}\label{count1}
n^2-\Lambda(A)-\frac{1}{2}\Lambda(P)\le \left( \frac{3}{5} + \left( \frac{2}{5}-\frac{\varepsilon}{2} \right) \varepsilon + \left( \frac{4}{5}-\frac{\beta}{2} \right) \beta \right) n^2 +O(n)
\end{align}
points which may be included in $S$; and all these points are below the line $x+y=\frac{8n}{5}+\ep n$. Then, writing
\[D_1:= \left\{ (x,y) \in [n]^2 \mid y > \frac{4n}{5}+\beta n,\; y+x<\frac{8n}{5}-\varepsilon n \right\} \] and
\[ D_2:= \left\{ (x,y) \in [n]^2 \mid x > \frac{4n}{5}+\beta n,\; y+x<\frac{8n}{5}-\varepsilon n \right\},\]
it follows from the assumption $|S|\ge(3/5-\delta)n^2$ and (\ref{count1}) that
\begin{align}\label{count2}
\frac{1}{n^2}|(D_1\cup D_2) \setminus S|\leq  \delta +\left( \frac{2}{5}-\frac{\varepsilon}{2} \right) \varepsilon + \left( \frac{4}{5}-\frac{\beta}{2} \right) \beta =:\upsilon(\delta,\varepsilon,\beta).
\end{align}
Note that we can choose $\delta,\varepsilon,\beta$ small enough such that $\upsilon(\delta,\varepsilon,\beta)=o(\gamma^2)$. In the remaining proof, we shall find in $D_1\cup D_2$ (or its translate) a relatively large subset of lattice points which are to be excluded from $S$, yielding a contradiction.
\begin{figure}[h]
    \centering
    \includegraphics[width = 90mm]{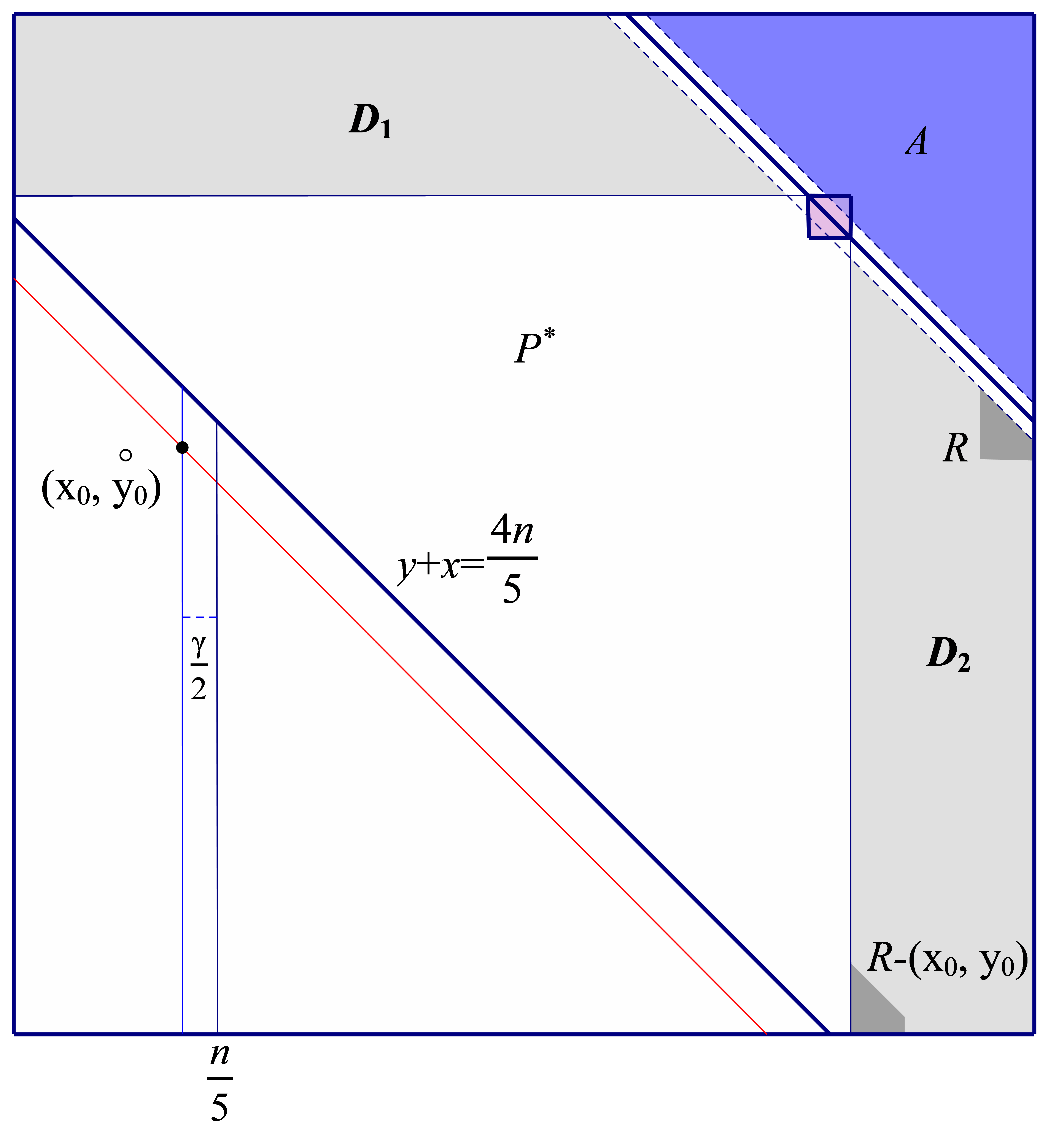}
    \caption{The case when $x_0< \frac{n}{5} - \frac{\gamma}{2} n$: The two grey regions $R:=(D_2 +p_0)\cap D_2$ and its translate $R-p_0$ form a pairing, which excludes from the sum-free set $S$ the amount of points which fit in one of the regions.}
    \label{fig6}
\end{figure}

First assume that $p_0$ is such that $x_0< \frac{n}{5} - \frac{\gamma}{2} n.$ Then the region $D_2 +p_0 $ intersects $D_2$ on a set of lattice points, denoted by $R$.
Since $R, R-p_0\subseteq D_1 \cup D_2$, applying Lemma \ref{lem-pairing} with $a=p_0$ and $T=R-p_0$ gives that $|(R\cup(R-p_0))\cap S|\le |R|$, and thus \[|(D_1\cup D_2) \setminus S|\ge|(R\cup(R-p_0))\setminus S|\ge |R\cup(R-p_0)|-|R|.\]
It is easy to observe that $|R\cup(R-p_0)|-|R|$ is minimized when $p_0$ is close to the point $(\frac{n}{5}-\frac{\gamma n}{2},\frac{3n}{5}-\frac{\gamma n}{2}),$ yielding an area of size at least $\left( \frac{3}{8}\gamma^2 + \frac{\gamma-2\beta}{4}(\beta-2\varepsilon) \right) n^2+O(n)$ (See Figure \ref{fig6}).
Thus $|(D_1\cup D_2) \setminus S|\ge \left( \frac{3}{8}\gamma^2 + \frac{\gamma-2\beta}{4}(\beta-2\varepsilon) \right) n^2+O(n)>\upsilon(\delta,\varepsilon,\beta)n^2$, a contradiction to (\ref{count2}).

Now it remains to consider the case when $p_0$ satisfies $x_0\geq \frac{n}{5} - \frac{\gamma n}{2}.$ We consider the overlap of $(D_1 \cup D_2) - p_0$ with $(x_1,y_1)-((D_1 \cup D_2)-p_0)$ and denote by $\mathcal{O}$ the set of lattice points in the overlap (see Figure \ref{overlaps-1}). Let
$$D:=((D_1 \cup D_2)-p_0)\setminus (D_1\cup D_2).$$
Then it is easy to verify that $\mathcal{O}\subseteq D$. Note that by Lemma \ref{lem-pairing} with $a=p_0$ and $T=D_1\cup D_2$, one has that
$$|(D\cup D_1 \cup D_2)\cap S|\le |D_1\cup D_2|.$$

\begin{figure}[h]
    \centering
    \includegraphics[width = 110mm]{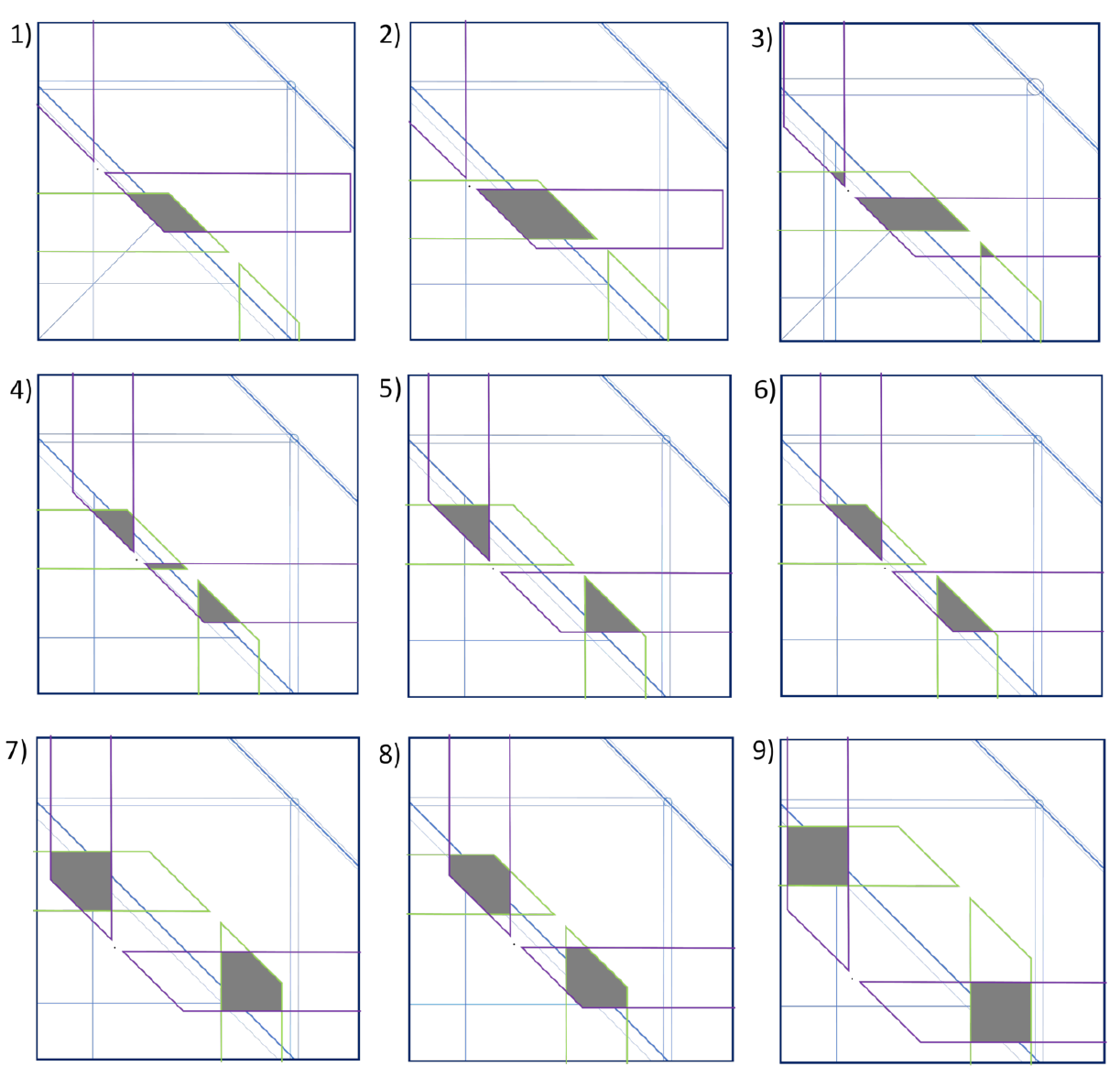}
    \caption{All possible shapes of $\mathcal{O}$: the green lines frame the region $(D_1 \cup D_2) - p_0$, whilst the purple lines frame $(x_1,y_1)-((D_1 \cup D_2)-p_0)$. The trivial cases where the overlap is cut off by the $x$- and $y$-axes are not shown.}
    \label{overlaps-1}
\end{figure}
Then, using~\eqref{count2}, we have
\begin{align*}
|\mathcal{O} \cap S|\le|D\cap S|&=|(D\cup D_1 \cup D_2)\cap S|-|(D_1 \cup D_2)\cap S|\\
&\le|(D_1\cup D_2) \setminus S| \leq \upsilon(\delta,\varepsilon,\beta)n^2.
\end{align*}
Moreover, by definition we know that $(x_1,y_1)-\mathcal{O}\subseteq \mathcal{O}$, that is,  $\mathcal{O}$ (and also $P\setminus\mathcal{O}$) is a pairing set for $(x_1,y_1)$. It follows from Lemma \ref{lem-pairing-set} that

\begin{align*}
    |S| &\leq n^2 - \Lambda(A) - \frac{1}{2}\Lambda(P\setminus\mathcal{O})-(|\mathcal{O}|-|\mathcal{O}\cap S|) \\
    &\le n^2 - \|A\|-\frac{1}{2} \|P\| - \frac{1}{2} |\mathcal{O}|+\upsilon(\delta,\varepsilon,\beta)n^2 \\
    &\le n^2 - \frac{1}{2}\Big(\frac{2n}{5}-\varepsilon n\Big)^2-\frac{1}{2}\Big(\frac{4n}{5}-\beta n\Big)^2- \frac{1}{2} |\mathcal{O}|+\upsilon(\delta,\varepsilon,\beta)n^2 \\
    &=\frac{3}{5}n^2+o(\gamma^2) n^2-\frac{1}{2}|\mathcal{O}|.
\end{align*}

Therefore, it suffices to show that $|\mathcal{O}|= \Omega(\gamma^2) n^2$, and in the remaining proof we shall verify this by considering all possible shapes of $\mathcal{O}$.
\begin{figure}[h]
    \centering
    \includegraphics[width=112mm]{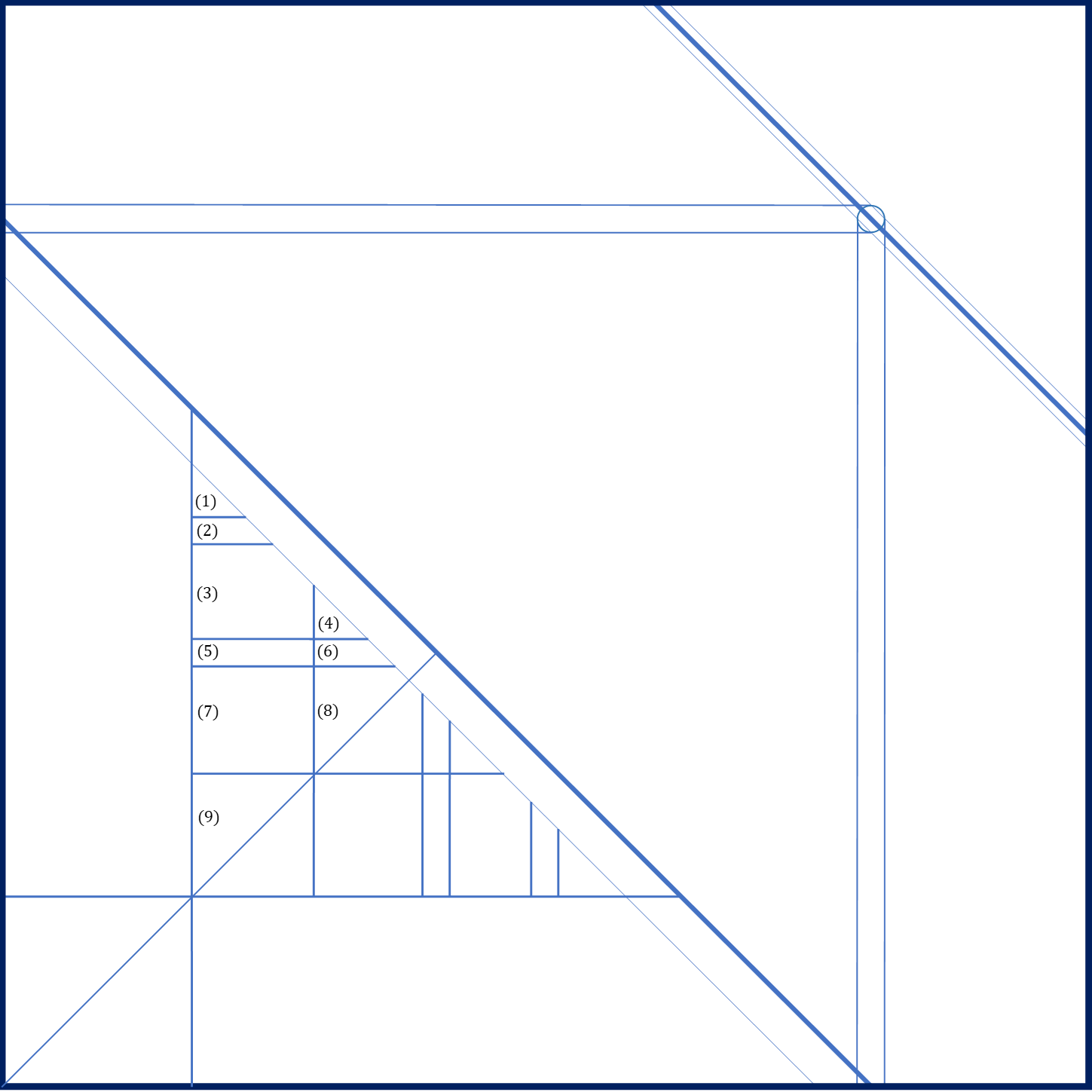}
    \caption{each numbered region will produce a unique shape of the overlap.}
    \label{F7}
\end{figure}

Since $(x_1,y_1)$ is $\beta$-close to $(\frac{4n}{5},\frac{4n}{5})$ and $\beta\ll\gamma$, we may further assume that $(x_1,y_1)=(\frac{4n}{5},\frac{4n}{5})$ in order not to cluster the presentation. We list in Figure \ref{overlaps-1} all possible shapes of the overlap $\mathcal{O}$, which originate from the location of the point $(x_0, y_0)$ (see Figure \ref{F7}). In particular, the area of the overlap in each of these cases is given as follows:
\begin{enumerate}
    \item[(1)] $4\left(\frac{3}{5}n-y_0\right) \left(\frac{4}{5}n- y_0- x_0- \varepsilon n \right), $ where $y_0\ge\frac{n}{2}+\frac{\beta}{2} n$, $x_0\ge\frac{1}{5}n-\frac{\gamma }{2}n$.
    \item[(2)] $4\left(y_0-\frac{2}{5}n-\beta n\right) \left(\frac{4}{5}n-y_0-x_0-\varepsilon n \right)$, where $y_0\in[\frac{1}{2}n-\frac{\beta+\varepsilon}{2} n, \frac{1}{2}n+\frac{\beta}{2} n]$, $x_0\ge\frac{1}{5}n-\frac{\gamma }{2}n$.
    \item[(3)] $\left(n-2y_0-\beta n -\varepsilon n \right)^2 + 4\left(y_0-\frac{2}{5}n-\beta n\right) \left(\frac{4}{5}n-y_0- x_0- \varepsilon n \right)$, where \\\[y_0\in\left[\frac{2}{5}n+\beta n,\frac{1}{2}n-\frac{\beta+\varepsilon}{2}n\right], x_0\in\left[
        \frac{1}{5}n-\frac{\gamma }{2}n,\frac{3}{10}n+\frac{\beta-\varepsilon}{2}n\right].\]
    \item[(4)] $4\left(\frac{4}{5}n-y_0- x_0- \varepsilon n \right) \left(x_0-\frac{1}{5}n-2\beta n\right)$, where $y_0\ge \frac{2}{5}n+\beta n,  x_0\ge\frac{3}{10}n+\frac{\beta-\varepsilon}{2} n$.
    \item[(5)] $\left( n-2y_0-\beta n-\varepsilon n \right)^2$, where \[y_0\in\left[\frac{2}{5}n-\frac{\varepsilon }{2}n, \frac{2}{5}n+\beta n\right], x_0\in\left[
        \frac{1}{5}n-\frac{\gamma }{2}n, \frac{3}{10}n+\frac{\beta-\varepsilon}{2}n\right].\]
    \item[(6)] $4 \left(\frac{4}{5}n-y_0- x_0- \varepsilon n \right) \left(\frac{1}{5}n+x_0-y_0-\beta n \right)$, where \[y_0\in \left[\frac{2}{5}n-\frac{\varepsilon }{2}n, \frac{2}{5}n+\beta n\right], x_0\ge\frac{3}{10}n+\frac{\beta-\varepsilon}{2} n.\]
    \item[(7)] $2 \left( \frac{1}{5}n-\beta n\right)^2 - \left(2y_0- \frac{3}{5}n+\varepsilon n - \beta n\right)^2$, where \[y_0\in\left[\frac{3}{10}n+\frac{\beta-\varepsilon}{2} n, \frac{2}{5}n-\frac{\varepsilon}{2} n\right], x_0\in\left[
        \frac{1}{5}n-\frac{\gamma }{2}n, \frac{3}{10}n+\frac{\beta-\varepsilon}{2}n\right].\]
    \item[(8)] $2 \left( \frac{1}{5}n-\beta n\right)^2 - \left(2y_0- \frac{3}{5}n+\varepsilon n - \beta n\right)^2-\left(2x_0- \frac{3}{5}n+\varepsilon n - \beta n\right)^2$, where \\ $$\frac{3}{10}n+\frac{\beta-\varepsilon}{2} n\le x_0\le y_0\le \frac{2}{5}n-\frac{\varepsilon }{2}n.$$
    \item[(9)] $2 \left( \frac{1}{5}n-\beta n\right)^2$, where $\frac{1}{5}n-\frac{\gamma }{2}n\le x_0\le y_0\le \frac{3}{10}n +\frac{\beta-\varepsilon}{2}n.$
\end{enumerate}
It is obvious that for the regions $5,$ $7$ and $9,$ the area of the overlap has size $\Omega(\gamma^2) n^2$. The only regions which interest us are the ones bordering the line $y+x=\frac{4n}{5}-\gamma n$. Moreover,
the regions in question are $1,$ $2,$ $3,$ $4,$ $6$ and $8.$ Among them, the minimum overlap is achieved in region $1$ by letting $(x_0,y_0)=(\frac{n}{5}-\frac{\gamma n}{2},\frac{3n}{5}-\frac{\gamma n}{2})$, which yields a value of $|\mathcal{O}|\ge 2\gamma(\gamma-\varepsilon)n^2$ as desired.
This completes the proof of Theorem \ref{mainthm}.
\end{proof}

\section{Proof of Theorem \ref{thm2}}\label{sec-pqSF}
In this section we investigate the maximum size of a $(p,p)$-sum-free set $S$. To simplify the presentation, we write $p$-sum-free for $(p,p)$-sum-free. Our proof builds on the techniques developed in the work of Elsholtz and Rackham \cite{E-R}.
We need a variant notion of pairing set as follows.
\begin{defn}
	For any $(a_1,a_2)\in\mathbb{R}^2_{[0,n]}$, $P\subseteq\mathbb{R}^2_{[0,n]}$ is a $p$-pairing set for $(a_1,a_2)$ if, for any $(x_1,x_2)\in P$, we have $(\frac{a_1}{p}-x_1,\frac{a_2}{p}-x_2)\in P$.
\end{defn}
Similar to Lemmas \ref{lem-pairing-set} and \ref{lem-pairing}, the following lemma guarantees that for any point $a\in S$ and its $p$-pairing set $P$, at least half of the points in $P$ are excluded from $S$. Similar statement also holds when we consider a set and its translate dilated by $p$. We omit the proof.

\begin{lemma}\label{lem1}
	Let $S\subseteq[n]^2$ be a $p$-sum-free set.
	\begin{itemize}
		\item[\emph{(1)}] If $P$ is a $p$-pairing set for some $a\in S$, then we have $|S\cap P|\leq\frac{1}{2}\Lambda(P)$.
		\item[\emph{(2)}] If $T\subseteq\mathbb{R}^2_{[0,n]}$ and $a\in S$, then $|S\cap (p(a+T)\cup T)|\leq\Lambda(T)$.
	\end{itemize}
\end{lemma}
\medskip
\begin{proof}[{ Proof of Theorem \ref{thm2}}]
	
	Let $S\subseteq [n]^2$ be a $p$-sum-free set. Our goal is to show that $|S|\leq\left(1-\frac{2}{4p^2+1}\right)n^2+O(n)$ for $p\ge 2$. We may neglect any boundary effects as they give error terms $O(n)$ for the size of $S$, which will be omitted so as to ease the presentation. We consider cases depending on the placement of upper boundary lines.
	
	\medskip
	\noindent \textbf{Case 1:} $|\partial S|\le1$. As vertices in the upper boundary come in (adjoint) pairs, we see that in this case $\partial S=\emptyset $, and thus Lemma \ref{onepoint} ensures the existence a point $p_1=(x_1,y_1)\in S$ such that $x_1\geq x$ and $y_1 \geq y$ for all $(x,y) \in S.$ Let $P:=\{(x,y)\mid   0\leq x\leq\frac{x_1}{p},0\leq y\leq\frac{y_1}{p}\}$. Then $P$ is a $p$-pairing set for $p_1$ and thus by Lemma \ref{lem1}, we have that \begin{align}
	|S| \nonumber
	&\leq(n+1)^2-(n-x_1)n-(n-y_1)x_1-\frac{1}{2}\Lambda(P)\\ \nonumber
	&=\left(1-\frac{1}{2p^2}\right)x_1y_1+O(n)\leq\left(1-\frac{1}{2p^2}\right)n^2+O(n)< \left(1-\frac{2}{4p^2+1}\right)n^2.\nonumber
	\end{align}
	\textbf{Case 2:} $|\partial S|\geq2$ and for every two points $p_1=(x_1,y_1),p_{2}=(x_2,y_2)$ that are adjoint in $\partial S$ with $x_1<x_2$ and $y_1>y_2$, we have either $m(p_1,p_2)>-\frac{y_{2}}{x_{2}}$ or $m(p_1,p_2)<-\frac{y_1}{x_1}$. \medskip
	
	In this case, we choose $p_1=(x_1, y_1)\in \partial S$ such that $x_1y_1\geq xy$ holds for every $(x, y)\in \partial S$ and $P_1:=\{(x,y)\mid   0\leq x\leq\frac{x_1}{p},0\leq y\leq\frac{y_1}{p}\}$. By symmetry, we may further assume that $y_1\ge x_1$. If there does not exist $p_2=(x_2, y_2)\in \partial S$ adjoint to $p_1$ with $x_2 > x_1$ and
	$y_2<y_1$, then by Lemma \ref{lem1} and that $y_1\ge x_1$, we have
	\begin{align*}
	|S| \leq n^2-(n-x_1)n-\frac{1}{2}\Lambda(P_1)\leq nx_1-\frac{x_1^2}{2p^2} \leq\left(1-\frac{1}{2p^2}\right)n^2.
	\end{align*}
	
	Thus, we may assume that there exists $p_2=(x_2, y_2)\in \partial S$ adjoint to $p_1$ with $x_2 > x_1$ and
	$y_2<y_1$. Let $L: y=mx+c$ be the line passing through $p_1,p_2$ and define $$A=\{(x,y)\in\mathbb{R}^2_{[0,n]}\mid  y>mx+c\}.$$
	We claim that $m<-\frac{y_1}{x_1}\leq-1$. Indeed, by the assumption of Case 2, assume for contradiction that $m>-\frac{y_2}{x_2}$, then
	\begin{equation*}
	x_2y_2
	=x_2(y_1+m(x_2-x_1))\geq x_2y_1-y_2(x_2-x_1)=x_1y_1+(y_1-y_2)(x_2-x_1)>x_1y_1,
	\end{equation*}
	contrary to the choice of $p_1$.
	
	We split into two subcases depending on the $x$- and $y$-intercept of $L$. Note first that, if $c\leq n$, then we have $-\frac{c}{m}\leq n$ because $m\le -1$, and so $|S|<\frac{1}{2}n^2$ as $A\cap S=\emptyset$.
	
	\textbf{(I).} If $c>n$ and $-\frac{c}{m}\leq n$, then
	\[|S| \leq n^2-\Lambda(A)-\frac{1}{2}\Lambda(P_1)=\frac{n}{m}\left(\frac{n}{2}-c\right)-\frac{1}{2p^2}x_1y_1=\frac{n}{m}\left(\frac{n}{2}-y_1\right)+x_1n-\frac{1}{2p^2}x_1y_1.\]
	Now if $y_1\leq\frac{n}{2}$, then as $m<-1$ and $x_1\le y_1\le \frac{n}{2}$, we observe that $|S|\leq x_1n\leq\frac{1}{2}n^2$. We may then assume $y_1>\frac{n}{2}$.
	
	If $x_1<\frac{n}{2}$, then by the assumption that $m<-\frac{y_1}{x_1}$, we have
	\begin{equation*}
	|S|
	\leq\frac{nx_1}{y_1}\left(y_1-\frac{n}{2}\right)+x_1n-\frac{1}{2p^2}x_1y_1
	\leq \left(2n-\frac{n^2}{2y_1}-\frac{y_1}{2p^2}\right)\frac{n}{2}\le\left(1-\frac{1}{2p}\right)n^2.
	\end{equation*}
	where the last inequality follows from $\frac{n^2}{2y_1}+\frac{y_1}{2p^2}\ge 2\sqrt{\frac{n^2}{2y_1}\frac{y_1}{2p^2}}=\frac{n}{p}$.
	
	Assume then $x_1\ge\frac{n}{2}$. Note that as $-\frac{c}{m}\le n$, the slope of $L$ is smaller than the slope of the line passing through $p_1$ and $(n,0)$, and so $m\leq\frac{-y_1}{n-x_1}$. Thus, we have
	\begin{align}
	|S| \nonumber
	&\leq\frac{n(n-x_1)}{y_1}\left(y_1-\frac{n}{2}\right)+x_1n-\frac{1}{2p^2}x_1y_1
	\leq n^2-\left(\frac{n-x_1}{2n}n^2+\frac{x_1y_1}{2p^2}\right)\\ \nonumber
	&=\frac{n^2}{2}+\left(\frac{n}{2}-\frac{y_1}{2p^2}\right)x_1 \le\frac{n^2}{2}+\left(\frac{n}{2}-\frac{y_1}{2p^2}\right)y_1 \le\left(1-\frac{1}{2p^2}\right)n^2, \nonumber
	\end{align}
	where the second last inequality follows since $x_1\le y_1$ and the last one follows from $p\ge 2$.
	
	\textbf{(II).} If $c>n$ and $-\frac{c}{m}>n$, then $A$ is a triangle and thus
	\begin{align}
	|S| \nonumber
	&\leq n^2-\Lambda(A)-\frac{1}{2}\Lambda(P_1)\\ \nonumber
	&=n^2+\frac{(n-y_1)^2}{2m}+\frac{m(n-x_1)^2}{2}-(n-x_1)(n-y_1)-\frac{x_1y_1}{2p^2}. \nonumber
	\end{align}
	The right-hand side above is increasing when $m\le -\frac{n-y_1}{n-x_1}$.
	Since $\frac{n-y_1}{n-x_1}\leq\frac{y_1}{x_1}\le-m$, it follows that
	\begin{align}
	|S| \nonumber
	&\leq n^2-\frac{(n-y_1)^2}{\frac{2y_1}{x_1}}-\frac{y_1(n-x_1)^2}{2x_1}-(n-x_1)(n-y_1)-\frac{x_1y_1}{2p^2}\\ \nonumber
	&\leq2(x_1+y_1)n-n^2-\left(2+\frac{1}{2p^2}\right)x_1y_1,  \nonumber
	\end{align}
	where the right-hand side of the last inequality is maximized when $x_1=y_1=\frac{4p^2}{4p^2+1}n$, and thus $|S|\leq\left(1-\frac{2}{4p^2+1}\right)n^2$.\medskip

	\noindent \textbf{Case 3:} There exist $p_1=(x_1,y_1),p_2=(x_2,y_2)$ adjoint in $\partial S$ such that $x_1<x_2$, $y_1>y_2$ and $-\frac{y_{1}}{x_{1}}\leq m(p_1,p_2)\leq-\frac{y_{2}}{x_{2}}$. \medskip
	
	For each $p_i$ with $i\in[2]$, define $P_i:=\{(x,y)\mid   0\leq x\leq\frac{x_i}{p},0\leq y\leq\frac{y_i}{p}\}$ and set $A=\{(x,y)\in\mathbb{R}^2_{[0,n]}\mid  y>mx+c\}$ (see Figure~\ref{lastF}). Since $m\leq-\frac{y_{2}}{x_{2}}$ and $y_2=mx_2+c$, we have that $y_2\leq\frac{c}{2}$. Similarly, by the condition $m\ge-\frac{y_{1}}{x_{1}}$, we have that $y_1\ge\frac{c}{2}$.
	
	Define $$T_1=\{(x,y)\in\mathbb{R}^2_{[0,n]}\mid   x\geq \frac{x_1}{p}, y\leq mx+\frac{c}{2p}\},$$
	and
	$$T_2=\{(x,y)\in\mathbb{R}^2_{[0,n]}\mid   y\geq \frac{y_2}{p}, y\leq mx+\frac{c}{2p}\}.$$ We claim that $T_1,T_2\neq\emptyset $. These amount to proving $-\frac{c}{2mp}\geq \frac{x_1}{p}$ and $\frac{c}{2p}\geq\frac{y_2}{p}$, which in turn follows from the fact that $y_2\leq\frac{c}{2}\le y_1$.
	\begin{figure}[h]
		\centering
		\includegraphics[width=250pt]{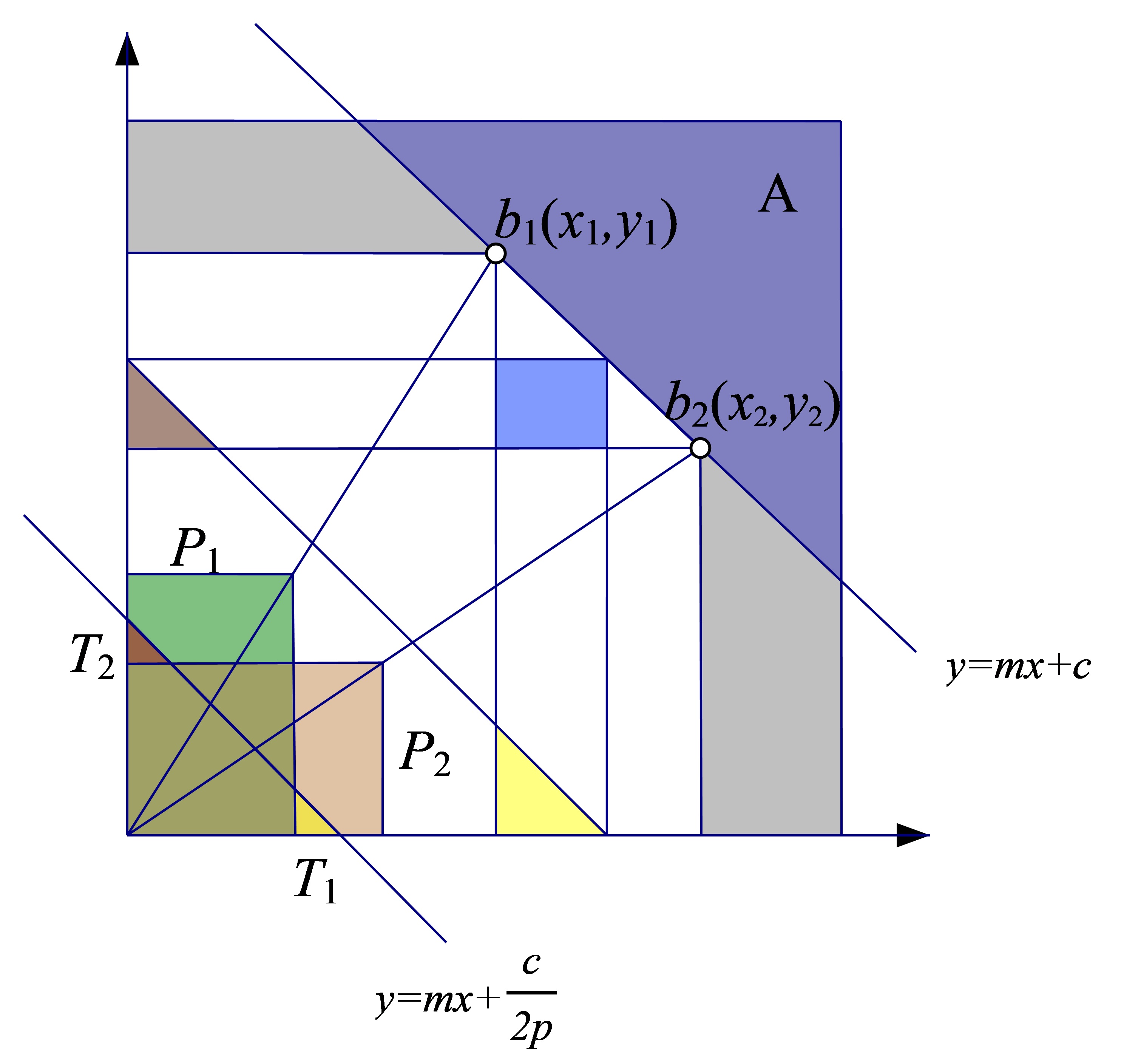}
		\caption{$T_1,T_2\neq\emptyset $.}
		\label{lastF}
	\end{figure}

	If $T_1\cap S=\emptyset $, then a short calculation shows
	\begin{align}
	|S| \nonumber
	\leq n^2-\Lambda(T_1)-\Lambda(A)-\frac{1}{2}\Lambda(P_1)\le n^2+\frac{c^2}{8p^2m}-\| A\|. \nonumber
	\end{align}
	
	If $T_1\cap S\neq\emptyset $, then take a point $a\in T_1\cap S$, then one can check that $p(a+T_2)\cap T_2=\emptyset$. By Lemma \ref{lem1}, we have $|S\cap(p(a+T_2)\cup T_2)|\leq\Lambda(T_2)$. By the definition of $T_2$, any point $(x,y)\in p(a+T_2)$
	satisfies that $y\leq mx+c$ and $x\geq x_1, y\geq y_2$. We again arrive to
	\begin{align}
	|S| \nonumber
	\leq n^2-\Lambda(T_2)-\Lambda(A)-\frac{1}{2}\Lambda(P_2)\leq n^2+\frac{c^2}{8p^2m}-\| A\|. \nonumber
	\end{align}
	
	Suppose now that $c\leq n$. If $-\frac{c}{m}\leq n$, then $|S|\leq\frac{1}{2}n^2$ by excluding $A$ alone. So $-\frac{c}{m}>n$. Then $\| A\|=\frac{n(2n-mn-2c)}{2}$ and we get, using $c\le n$ and $x+y\ge 2\sqrt{xy}$ for $x,y>0$,
	\begin{align}
	|S| \nonumber
	&\leq n^2+\frac{c^2}{8p^2m}-\| A\| =\frac{c^2}{8p^2m}+\frac{n^2m}{2}+cn\leq cn-\frac{cn}{2p}\le\left(1-\frac{1}{2p}\right)n^2.  \nonumber
	\end{align}

	We may then assume $c>n$. The case $-\frac{c}{m}\le n$ can be handled as the above $c\le n$ and $-\frac{c}{m}\ge n$ case. Thus, we can assume $-\frac{c}{m}>n$. Then $A$ is a triangle with $\| A\|=-\frac{(n-mn-c)^2}{2m}$ and
	\begin{align*}
	|S|
	\leq n^2+\frac{c^2}{8p^2m}+\frac{(n-mn-c)^2}{2m} =\left(\frac{1}{8p^2}+\frac{1}{2}\right)\frac{c^2}{m}+\frac{n(m-1)}{m}c+\frac{n^2m}{2}+\frac{n^2}{2m}.
	\end{align*}
	The quadratic function of $c$ above is maximized when $c=-\frac{(m-1)n}{1+\frac{1}{4p^2}}$. Thus
	\begin{align}
	|S| \nonumber
	&\leq n^2\left[\frac{4p^2}{4p^2+1}+\left(\frac{1}{2}-\frac{2p^2}{4p^2+1}\right)\left(m+\frac{1}{m}\right)\right] \leq\left(1-\frac{2}{4p^2+1}\right)n^2, \nonumber
	\end{align}
	where the maximum is achieved when we choose $m=-1$ and thus $c=\frac{8p^2}{4p^2+1}n$.\medskip
	
	This completes the proof.
\end{proof}
\bibliographystyle{amsalpha}
\bibliography{ref}

\providecommand{\bysame}{\leavevmode\hbox to3em{\hrulefill}\thinspace}
\providecommand{\MR}{\relax\ifhmode\unskip\space\fi MR }
\providecommand{\MRhref}[2]{%
  \href{http://www.ams.org/mathscinet-getitem?mr=#1}{#2}
}
\providecommand{\href}[2]{#2}
\begin{thebibliography}{DFST99}

\bibitem[Blo16]{bloom16}
T.~F. Bloom, \emph{A quantitative improvement for {R}oth's theorem on
  arithmetic progressions}, J. Lond. Math. Soc. (2) \textbf{93} (2016), no.~3,
  643--663. \MR{3509957}

\bibitem[BLST18]{BLST}
J\'{o}zsef Balogh, Hong Liu, Maryam Sharifzadeh, and Andrew Treglown,
  \emph{Sharp bound on the number of maximal sum-free subsets of integers}, J.
  Eur. Math. Soc. (JEMS) \textbf{20} (2018), no.~8, 1885--1911. \MR{3854894}

\bibitem[Cam02]{cam2}
Peter~J. Cameron, \emph{Sum-free sets of a square}, manuscript (2002).

\bibitem[Cam05]{cameron19}
\bysame, \emph{Research problems from the 19th {B}ritish {C}ombinatorial
  {C}onference}, Discrete Math. \textbf{293} (2005), no.~1-3, 313--320.
  \MR{2136071}

\bibitem[CKP20]{cho}
Ilkyoo Choi, Ringi Kim, and Boram Park, \emph{Maximum {$k$}-sum n-free sets of
  the 2-dimensional integer lattice}, Electron. J. Combin. \textbf{27} (2020),
  no.~4, Paper No. 4.2, 12. \MR{4245177}

\bibitem[DFST99]{DE99}
Jean-Marc Deshouillers, Gregory~A. Freiman, Vera S\'{o}s, and Mikhail Temkin,
  \emph{On the structure of sum-free sets. {II}}, no. 258, 1999, Structure
  theory of set addition, pp.~xii, 149--161. \MR{1701193}

\bibitem[ER17]{E-R}
Christian Elsholtz and Laurence Rackham, \emph{Maximal sum-free sets of integer
  lattice grids}, J. Lond. Math. Soc. (2) \textbf{95} (2017), no.~2, 353--372.
  \MR{3656272}

\bibitem[Fre92]{Freiman92}
Gregory~A. Freiman, \emph{On the structure and the number of sum-free sets},
  no. 209, 1992, Journ\'{e}es Arithm\'{e}tiques, 1991 (Geneva), pp.~13,
  195--201. \MR{1211012}

\bibitem[HT17]{HAN17}
Robert Hancock and Andrew Treglown, \emph{On solution-free sets of integers},
  European J. Combin. \textbf{66} (2017), 110--128. \MR{3692141}

\bibitem[Rot53]{roth53}
K.~F. Roth, \emph{On certain sets of integers}, J. London Math. Soc.
  \textbf{28} (1953), 104--109. \MR{51853}

\bibitem[Ruz93]{ru1}
Imre~Z. Ruzsa, \emph{Solving a linear equation in a set of integers. {I}}, Acta
  Arith. \textbf{65} (1993), no.~3, 259--282. \MR{1254961}

\bibitem[Ruz95]{ru2}
\bysame, \emph{Solving a linear equation in a set of integers. {II}}, Acta
  Arith. \textbf{72} (1995), no.~4, 385--397. \MR{1348205}

\bibitem[Sch16]{SCHUR}
I.~Schur, \emph{Uber die kongruenz $x^m + y^m \equiv z^m$ (mod p)}, Jahresber.
  Deutsch. Math.-Verein. \textbf{25} (1916), 114--117.

\bibitem[Skr93]{Skri}
Maxim Skriganov, \emph{On integer points in polygons}, Ann. Inst. Fourier
  (Grenoble) \textbf{43} (1993), no.~2, 313--323. \MR{1220271}

\bibitem[Tra18]{Tuan18}
Tuan Tran, \emph{On the structure of large sum-free sets of integers}, Israel
  J. Math. \textbf{228} (2018), no.~1, 249--292. \MR{3874843}

\end{thebibliography}
\end{document}